\documentclass[reqno]{amsart}
\usepackage{bm,amscd,textcomp,mathrsfs,lscape}
\usepackage{amsmath,amsthm,amssymb,verbatim,mathtools}
\usepackage{color}
\usepackage{enumerate}
\usepackage[dvipdfm]{graphicx}
\usepackage{latexsym}
\usepackage[all]{xy}
\usepackage{calc}

\DeclareMathOperator{\Hom}{Hom}

\newtheorem{theorem}{Theorem}[section]
\newtheorem{lemma}[theorem]{Lemma}
\newtheorem{proposition}[theorem]{Proposition}
\newtheorem{corollary}[theorem]{Corollary}
\theoremstyle{remark}

\theoremstyle{definition}

\numberwithin{equation}{section}

\title[Ordinary quivers of Hochschild extension algebras]
	{The ordinary quivers of Hochschild extension algebras for self-injective Nakayama algebras}
\author[H. Koie]{Hideyuki Koie}
\address[H. Koie]{Department of Mathematics,
					Tokyo University of Science, 1-3 Kagurazaka,
                    	Shinjuku-ku, Tokyo 162--8601, Japan}
\email{1114702@ed.tus.ac.jp}

\author[T. Itagaki]{Tomohiro Itagaki}
\address[T. Itagaki]{Department of Mathematics,
					Tokyo University of Science, 1-3 Kagurazaka,
                    	Shinjuku-ku, Tokyo 162--8601, Japan}
\email{titagaki@rs.tus.ac.jp}

\author[K. Sanada]{Katsunori Sanada}
\address[K. Sanada]{Department of Mathematics,
					Tokyo University of Science, 1-3 Kagurazaka,
                    	Shinjuku-ku, Tokyo 162--8601, Japan					}
\email{sanada@rs.tus.ac.jp}

\date{}

\begin{document}
\begin{abstract}
Let $T$ be a Hochschild extension algebra of a finite dimensional algebra $A$ over a field $K$
by the standard duality $A$-bimodule $\Hom_K(A,\,K)$.
In this paper, we determine the ordinary quiver of $T$
if $A$ is a self-injective Nakayama algebra
by means of the $\mathbb{N}$-graded second Hochschild homology group $HH_2(A)$
in the sense of Sk\"oldberg.
\end{abstract}
\maketitle
%
%
%
%
\section{Introduction}
Throughout the paper,
an algebra means a finite dimensional algebra
over a field $K$.
By a Hochschild extension of an algebra $A$ by a duality module $M$,
we mean an exact sequence
$$
	0 \longrightarrow
    	M \stackrel{\kappa}{\longrightarrow}
        	T \stackrel{\rho}{\longrightarrow}
            	A \longrightarrow
                	0
$$
such that $T$ is a $K$-algebra, $\rho$ is an algebra epimorphism
and $\kappa$ is a $T$-bimodule monomorphism.
In the above, $M$ is an $A$-bimodule, so $M$ is regarded as a $T$-bimodule
by means of $\rho$.
Hochschild \cite{Hochschild} proved that the set of equivalent classes of Hochschild extensions
over $A$ by $M$
is in one-to-one correspondence with the second Hochschild cohomology group $H^2(A,\,M)$.
If $M$ is the standard duality module $D(A) = \Hom_K(A,\,K)$,
then we denote by $T_\alpha(A)$ the Hochschild extension algebra corresponding to a $2$-cocycle
$\alpha : A \times A \rightarrow D(A)$.
Then, $T_0(A)$ is just the trivial extension algebra
$A \ltimes D(A)$.
Hochschild extension algebras and trivial extension algebras
play an important role in the representation theory of self-injective algebras
(e.g. \cite{yamagata 1988, Handbook of algebra}).
It is well known that
trivial extension algebras are symmetric,
Hochschild extension algebras are self-injective
and are not symmetric in general (see \cite{yamagata}).
If $T_\alpha(A)$ is not symmetric, then $T_\alpha(A)$ and
$A \ltimes D(A)$ are not Morita equivalent.
Nevertheless, Yamagata showed that 
if the ordinary quiver of $A$ has no oriented cycles,
then $T_\alpha(A)$ and $A \ltimes D(A)$
are related by a socle equivalence which naturally induces
a stable equivalence (see \cite{Ohnuki, yamagata 1988}).
We are interested in the ordinary quiver and the relations for Hochschild extension
algebras.
Although, in \cite{Fernandez}, Fernandez and Platzeck gave the ordinary quiver and the relations
for the trivial extension algebras,
it seems that there is little information about the ordinary quivers for general
Hochschild extension algebras.

In this paper, we consider Hochschild extension algebras of self-injective Nakayama algebras
by the standard duality module.
A basic Nakayama algebra is of the form $K\Delta/I$,
where $\Delta$ is a cyclic quiver and $I$ is an admissible ideal of $K\Delta$.
In addition, the Nakayama algebra is self-injective if and only if $I$ is generated by 
the paths of length at least $l$ for some $l \geq 2$.
The aim of the present paper is to determine the ordinary quivers of Hochschild extension
algebras for self-injective Nakayama algebras.

This paper is organized as follows:
In Section $2$, we recall some definitions and elementary facts about
Hochschild extension algebras and Hochschild cohomology groups.
In Section $3$, following Sk\"oldberg \cite{skoldberg},
we describe the projective resolution \textrm{\boldmath $P$}
for the truncated quiver algebra $A$.
Sk\"oldberg computed the Hochschild homology group of $A$
using the facts that
the complex $A \otimes_{A^e} \textrm{\boldmath $P$}$ is $\mathbb{N}$-graded
and, therefore $HH_n(A)$ is also $\mathbb{N}$-graded for each $n \geq 0$,
that is, $HH_n(A) = \bigoplus_{q=0}^\infty HH_{n,\,q}(A)$.
On the other hand, in \cite{Cibils},
Cibils gave a useful projective resolution \textrm{\boldmath $Q$}
for more general algebras.
Furthermore, in \cite{ACT}, Ames, Cagliero and Tirao gave a chain map
from \textrm{\boldmath $Q$} to \textrm{\boldmath $P$}.
Using this chain map, we give the isomorphism between $\bigoplus_{q} D(HH_{2,\,q}(A))$
and $H^2(A,\, D(A))$.
These grading and the comparison morphism are very useful.
Itagaki and Sanada \cite{Itagaki} and Vol$\check{{\rm c}}$i$\check{{\rm c}}$ \cite{Volcic}
used the ideas for the computation of cyclic homology.

In Section $4$,
we consider the ordinary quiver of Hochschild extension algebras of self-injective Nakayama
algebras.
First of all, we have $2$-cocycles $A \times A \rightarrow D(A)$ through the above isomorphism
$\bigoplus_q D(HH_{2,\,q}(A)) \cong H^2(A,\, D(A))$.
Next, we give
a sufficient condition on a $2$-cocycle $\alpha$
for the ordinary quiver of ${T_\alpha(A)}$ to coincide with
that of the trivial extension algebra $A \ltimes D(A)$.
Then, we show that, for a self-injective Nakayama algebra $A$ and a $2$-cocycle
$\alpha : A \times A \rightarrow D(A)$,
the ordinary quiver of ${T_\alpha(A)}$ coincides with
either the ordinary quiver of $A$ or that of the trivial extension algebra $A \ltimes D(A)$.
Moreover, we give a sufficient condition for Hochschild
extension algebras of self-injective Nakayama algebras to be symmetric.
In Section $5$,
we exhibit some examples to clarify the theorem.

For general facts on quivers, we refer to \cite{ASS} and \cite{Fro},
and for Hochschild extension algebras,
we refer to \cite{Hochschild} and \cite{Handbook of algebra}.
%
%
\section{Hochschild extension algebras}
In 1945, the cohomology theory and extension theory of associative  algebras were
introduced by Hochschild \cite{Hochschild} (and also see \cite{Handbook of algebra}).
Let us recall his construction here.
Let $A$ be an algebra over a field $K$, and $D$ the standard duality functor $\Hom_K(-,\,K)$.
Note that $D(A) = \Hom_k (A,\, K)$ is an $A$-bimodule.
An {\it extension over $A$ with kernel $D(A)$} is an exact sequence
$$
	0 \longrightarrow
    	D(A) \stackrel{\kappa}{\longrightarrow}
        	T \stackrel{\rho}{\longrightarrow}
            	A \longrightarrow
                	0
$$
such that $T$ is a $K$-algebra, $\rho$ is an algebra epimorphism 
and $\kappa$ is a $T$-bimodule monomorphism from ${}_{\rho}(D(A))_{\rho}$.
Then $T$ is called a {\it Hochschild extension algebra of $A$ by $D(A)$},
and an extension is also called a {\it Hochschild extension}.
If we identify $D(A)$ with ${\rm Ker}\, \rho$, then $D(A)$ is a two-sided ideal of $T$ with
$(D(A))^2 = 0$ and the factor algebra $T/D(A)$ is isomorphic to $A$ by $\rho$.
Since $D(A)$ is nilpotent in $T$, a set of orthogonal idempotents
$\{ e_1, \ldots, e_l \}$ with $1_A = \sum_{i=1}^l e_i$
can be lifted to a set of orthogonal idempotents $\{ \textbf{e}_1, \ldots, \textbf{e}_l \}$ of $T$
so that
$1_T = \sum_{i=1}^l \textbf{e}_i$
and $e_i = \rho(\textbf{e}_i)$ for any $1 \leq i \leq l$.

%
%
A Hochschild extension algebra corresponds to a $2$-cocycle
$A \times A \rightarrow D(A)$.
Let us recall the definition of $2$-cocycle here.
For a $K$-algebra $A$, its Hochschild cohomology groups
$
	H^n(A,\,D(A))
$
are defined as
$
	{\rm Ext}_{A^e}^n(A,\,D(A))
$
for $n \geq 0$,
where $A^e = A \otimes_K A^{op}$ denotes the enveloping algebra of $A$.
There is a well known projective resolution of $A$ as a left $A^e$-module
$$
    {\rm Bar}_*(A):\cdots\stackrel{}{\longrightarrow}
    	A^{\otimes (n+2)}\stackrel{\delta_n}{\longrightarrow}
      		A^{\otimes (n+1)}\stackrel{}{\longrightarrow}
        		\cdots \stackrel{\delta_1}{\longrightarrow}
        			A^{\otimes 2}\stackrel{\delta_0}{\longrightarrow}
          				A \stackrel{}{\longrightarrow}
                        	0,
$$
where 
$
A^{\otimes (n+2)} := A \otimes_K A \otimes_K \cdots \otimes_K A
$
($n+2$ fold tensor products over $K$),
$
	\delta_0 : A \otimes_K A \rightarrow A
$
is the multiplication map
and, for $n \geq 0$, $\delta_n$ is defined by
$$
	\delta_n(a_0 \otimes a_1 \otimes \cdots \otimes a_{n+1})
  		= \sum_{i=0}^n (-1)^i a_0 \otimes \cdots \otimes a_{i-1}
    				\otimes a_i a_{i+1} \otimes a_{i+2} \otimes \cdots \otimes a_{n+1}.
$$
This resolution is called a {\it bar resolution} of $A$.
The complex 
$C^*(A) = \Hom_{A^e}({\rm Bar}_*(A),$ $D(A))$
is used to compute $H^n(A,\,D(A))$.
Note that, for each $n \geq 1$, there is a natural isomorphism:
\begin{equation}\label{mu}
	\mu_n : C^n(A) = \Hom_{A^e}(A^{\otimes (n+2)},\,D(A)) \xrightarrow{\sim} \Hom_{K}(A^{\otimes n},\,D(A)).
\end{equation}
We identify $C^0(A)$ with $D(A)$.
Thus, through $\mu_*$, $C^*(A)$ has the following form:
\begin{align*}
    C^*(A):0 \stackrel{}{\longrightarrow}
    	D(A) \stackrel{\overline{\delta}_1}{\longrightarrow}
    		&\Hom_K(A,\,D(A)) \stackrel{\overline{\delta}_2}{\longrightarrow}
      			\Hom_K(A\otimes A,\,D(A)) \stackrel{}{\longrightarrow}
                \cdots\\ \stackrel{}{\longrightarrow}
        			&\Hom_K(A^{\otimes n},\,D(A)) \stackrel{\overline{\delta}_{n+1}}{\longrightarrow}
        				\Hom_K(A^{\otimes (n+1)},\,D(A)) \stackrel{}{\longrightarrow}
          					\cdots.
\end{align*}
Here, for $\alpha \in \Hom_K(A^{\otimes n},\,D(A))$, $\overline{\delta}_{n+1}(\alpha)$
is defined by sending
$a_1 \otimes \cdots \otimes a_{n+1}$ to
\begin{align*}
	a_1 &\alpha(a_2 \otimes \cdots \otimes a_{n+1})
    		+ \sum_{i=1}^n (-1)^i 
            \alpha(a_1 \otimes \cdots \otimes a_{i-1} \otimes a_{i}a_{i+1} \otimes a_{i+2} \otimes \cdots a_{n+1}) \\
            		&+ (-1)^{n+1}\alpha(a_{1}\otimes \cdots \otimes a_{n}) a_{n+1}.
\end{align*}

A $2$-cocycle $\alpha \in {\rm Z}^2(A,\,D(A)) := {\rm Ker}\, \overline{\delta}_3$
defines the $K$-bilinear map $A \times A \rightarrow D(A)$ by the composition of
the map $A \times A \rightarrow A \otimes A;\, (a,\,b) \mapsto a \otimes b$ and $\alpha$.
We denote that by $\alpha$ again.
Note that $\alpha$ satisfies the relation
$$
	a_{1}\alpha(a_{2},\,a_{3}) 
    	- \alpha(a_{1}a_{2},\,a_{3})
        	+ \alpha(a_{1},\,a_{2}a_{3})
            	- \alpha(a_{1},\,a_{2})a_{3}
                	= 0
$$
for all $a_{1},\,a_{2},\,a_{3} \in A$.
Using a $2$-cocycle $\alpha$, we define an associative multiplication in the $K$-vector space
$A \oplus D(A)$ by the rule:
\begin{equation}\label{multi}
	(a,\,x)(b,\,y) = (ab,\, ay + xb + \alpha(a,\,b))
\end{equation}
for $(a,\,x),\,(b,\,y) \in A \oplus D(A)$.
Then it is easy to see that
$T_\alpha(A) := A \oplus D(A)$ is an associative $K$-algebra with identity $(1_A,\, -\alpha(1,\,1))$,
and that there exists an extension over $A$ by $D(A)$:
$$
	0 \longrightarrow
    	D(A) \longrightarrow
        	T_\alpha(A) \longrightarrow
            	A \longrightarrow
                	0.
$$
Conversely, given an extension over $A$ by $D(A)$;
$$
	0 \longrightarrow
    	D(A) \longrightarrow
        	T \longrightarrow
            	A \longrightarrow
                	0,
$$
we easily see that $T$ is isomorphic to $T_\alpha(A)$
for some $2$-cocycle $\alpha \in {\rm Z}^2(A, \, D(A))$.

%
%
By definition, two extensions $(F)$ and $(F')$ over $A$ with kernel $D(A)$ are {\it equivalent}
if there exists a $K$-algebra homomorphism $\iota:T \rightarrow T'$ such that the diagram
$$
  	\xymatrix{
  		(F) & 0 \ar[r] & D(A) \ar[r] \ar[d]^1 & T \ar[r] \ar[d]^\iota & A \ar[r] \ar[d]^1 & 0 \\
        (F')& 0 \ar[r] & D(A) \ar[r]          & T' \ar[r]           & A \ar[r]          & 0 
  	}
$$
is commutative.
The set of all equivalent classes of extensions over $A$ by $D(A)$
is denoted by $F(A,\,D(A))$.
An extension
$$
	0 \longrightarrow
    	D(A) \stackrel{\kappa}{\longrightarrow}
        	T \stackrel{\rho}{\longrightarrow}
            	A \longrightarrow
                	0
$$
is said to be {\it splittable} if there is an algebra homomorphism
$\rho' : A \rightarrow T$ with $\rho\rho' = {\rm id}_A$.
The following result is well known.
\begin{proposition}[{See \cite[Proposition 6.2]{Hochschild}, \cite[Section 2.5]{Handbook of algebra}}]
	The set $F(A,\,D(A))$ is in a one-to-one correspondence with $H^2(A,\,D(A))$.
	This correspondence $H^2(A,$ $D(A))$ $\rightarrow F(A,\,D(A))$ is obtained by assigning to each
	$2$-cocycle $\alpha \in {\rm Z}^2(A,\,D(A))$, the extension given by the multiplication {\rm (\ref{multi})}.
    The zero element in $H^2(A,\,D(A))$ is correspond to the class of splittable extensions.
\end{proposition}
Note that,
$T_0(A)$ is called a trivial extension algebra of $A$ by $D(A)$.

Next, we consider the Hochschild homology $HH_*(A):= H_*(A,\,A)$,
which coincides with the homology of the complex $C_*(A) := A \otimes_{A^e} {\rm Bar}_*(A)$.
Notice that
$C_n(A) = A \otimes_{A^e} A^{\otimes(n+2)} \cong A^{\otimes(n+1)}$
and the differential
$
	\tilde{\delta}_n : C_n(A)=A^{\otimes(n+1)} \rightarrow C_{n-1}(A)=A^{\otimes n}
$
sends $a_0 \otimes \cdots \otimes a_n$ 
to
$$
	\sum_{i=1}^{n-1}(-1)^i 
    	a_0 
        	\otimes \cdots \otimes 
            	a_{i-1}
                	\otimes
                    	a_i a_{i+1}
                        	\otimes
                            	a_{i+2}
                                	\otimes \cdots \otimes 
                                    	a_n
    + (-1)^n
    	a_n a_0
        	\otimes
            	a_1
                	\otimes \cdots \otimes
                    	a_{n-1}.
$$
The standard duality induces the isomorphism
$H^*(A,\,D(A)) \cong D(HH_*(A))$.
In fact, we have the following isomorphism and equations of complexes:
\begin{align*}
	D(C_*(A))
    	&=\Hom_K(C_*(A),\,K) \\
    		&= \Hom_K(A \otimes_{A^e} {\rm Bar}_*(A),\,K) \\
        		&\cong \Hom_{A^e}({\rm Bar}_*(A),\, \Hom_k(A,\,K)) \\
            		&= \Hom_{A^e}({\rm Bar}_*(A),\, D(A)).
\end{align*}
%
%
\section{The projective resolution by Sk\"oldberg and one by Cibils}
Let $\Delta$ be a finite quiver and $K$ a field.
We fix a positive integer $n \geq 2$.
A {\it truncated quiver algebra} is defined by $K\Delta/R_\Delta^n$,
where $R_\Delta$ is the arrow ideal of $K\Delta$
and $R_\Delta^n$ is the two-sided ideal of $K\Delta$ generated by the paths of length $n$.
We denote by $\Delta_0$, $\Delta_1$ and $\Delta_i$ the set of vertices,
the set of arrows and the set of paths of length $i$, respectively.
We put $\Delta_+ = \bigcup_{i=1}^{\infty} \Delta_i$.
Let $K\Delta_i$ be the $K$-vector space whose basis is the elements of $\Delta_i$
and $K\Delta_0$ the subalgebra of $K\Delta/R_\Delta^n$ generated by $\Delta_0$.
Moreover,
$K\Delta$ is $\mathbb{N}$-graded, that is, $K\Delta = \bigoplus_{i=0}^{\infty}K\Delta_i$.
For a path $x$, we denote by $s(x)$ and $t(x)$ primitive idempotents
with $x = s(x)xt(x)$ in $K\Delta$.

In \cite{skoldberg}, the Hochschild homology of a truncated quiver algebra was computed.
We recall the projective resolutions of a truncated quiver algebra \textrm{\boldmath $P$}
and \textrm{\boldmath $Q$} constructed by Sk\"oldberg and Cibils, respectively.
Moreover, we describe a chain map from \textrm{\boldmath $Q$} to \textrm{\boldmath $P$}.
We only show the third and under terms of each resolutions.
\begin{theorem}[{See \cite[Theorem 1]{skoldberg}}]\label{Skoldberg resolution}
    Let $A$ be a truncated quiver algebra $K\Delta/R_\Delta^n$.
    Then we have the following projective resolution of $A$ as a left $A^e$-module{\rm :}
    \begin{align*}
    \textrm{\boldmath $P$}:\cdots \stackrel{}{\longrightarrow}
    		A\otimes_{K\Delta_0}&K\Delta_{n+1}\otimes_{K\Delta_0}A \stackrel{d_3}{\longrightarrow}
        		A\otimes_{K\Delta_0}K\Delta_{n}\otimes_{K\Delta_0}A \\\stackrel{d_2}{\longrightarrow}
                	&A\otimes_{K\Delta_0}K\Delta_{1}\otimes_{K\Delta_0}A \stackrel{d_1}{\longrightarrow}
          				A\otimes_{K\Delta_0}A \stackrel{d_0}{\longrightarrow}
                        	A \stackrel{}{\longrightarrow}
                            	0.
    \end{align*}
    Here the differentials $d_2$ and $d_3$ are defined by
    $$
    	d_2(x \otimes y_1 \cdots y_n \otimes z)
    		= \sum_{j=0}^{n-1} x \otimes y_1 \cdots y_j \otimes y_{j+1} \otimes y_{j+2} \cdots y_n z
    $$
    and
    $$
    	d_3(x \otimes y_1 \cdots y_{n+1} \otimes z)
        	= x y_1 \otimes y_2 \cdots y_{n+1} \otimes z 
            	- x \otimes y_1 \cdots y_n \otimes y_{n+1} z,
    $$
    for $x,\,z \in A$ and $y_i \in \Delta_1$ $(1 \leq i \leq n+1)$.
\end{theorem}
%
%
We denote by $\textrm{\boldmath {$P$}}_i$
the $i$\,th term of $\textrm{\boldmath {$P$}}$,
then $A \otimes_{A^e} \textrm{\boldmath {$P$}}_i$
is the $i$\,th term of $A \otimes_{A^e} \textrm{\boldmath{$P$}}$.
We have
\begin{align}
	A \otimes_{A^e} \textrm{\boldmath{$P$}}_1
     &= A \otimes_{A^e} (A \otimes_{K\Delta_0} K\Delta_1 \otimes_{K\Delta_0} A) \notag \\
    	&\xrightarrow{\sim} A \otimes_{A^e} A^e \otimes_{K\Delta_0^e} K\Delta_1
        	\xrightarrow{\sim} A \otimes_{K\Delta_0^e} K\Delta_1. \label{iso}
\end{align}
Through the above isomorphisms (\ref{iso}),
we define the degree $q\,(\geq 1)$ part of $A \otimes_{A^e} \textrm{\boldmath{$P$}}_1$
in the following way:
$$
	(A \otimes_{A^e} \textrm{\boldmath{$P$}}_1)_q
    	= (A \otimes_{K\Delta_0^e} K\Delta_1)_q
        = \bigoplus_{\substack{
        				a_i \in \Delta_1 \, (1 \leq i \leq q) \\
                        \text{s.t. $t(a_{q-1})=s(a_q)$,} \\
        				t(a_q)=s(a_1) 
        			}
        }
        		K(a_1 \cdots a_{q-1} \otimes a_q).
$$
In a similar way,
we define the degree $q$ part of $A \otimes_{A^e} \textrm{\boldmath{$P$}}_2$
and $A \otimes_{A^e} \textrm{\boldmath{$P$}}_3$ by
\begin{align*}
	(A \otimes_{A^e} \textrm{\boldmath{$P$}}_2)_q
    	&=(A \otimes_{K\Delta_0^e} K\Delta_n)_q \\
        &= \bigoplus_{\substack{
        				a_i \in \Delta_1 \, (1 \leq i \leq q)\\
                        \text{s.t. $t(a_{q-n})=s(a_{q-n+1})$},\\
        				t(a_q)=s(a_1)
                        }}
        		K(a_1 \cdots a_{q-n} \otimes a_{q-n+1} \cdots a_{q}), \\
	(A \otimes_{A^e} \textrm{\boldmath{$P$}}_3)_q
    	&= (A \otimes_{K\Delta_0^e} K\Delta_{n+1})_q \\
        &= \bigoplus_{\substack{
        				a_i \in \Delta_1 \, (1 \leq i \leq q) \\
                        \text{s.t. $t(a_{q-n-1})=s(a_{q-n})$,}\\
        				t(a_q)=s(a_1) 
                        }}
        		K(a_1 \cdots a_{q-n-1} \otimes a_{q-n} \cdots a_{q}).
\end{align*}
Then
$\tilde{d}_{i} = {\rm id} \otimes d_i$
preserves $\mathbb{N}$-grading.
Actually,
for $i = 2,\,3$,
$
	(\tilde{d}_{i})_q :
    	(A \otimes_{A^e} \textrm{\boldmath{$P$}}_i)_q
             \rightarrow
             	(A \otimes_{A^e} \textrm{\boldmath{$P$}}_{i-1})_q
$
is given by
\begin{align}
	(\tilde{d}_{2})_q(a_1 \cdots a_{q-n} \otimes a_{q-n+1} \cdots a_{q}) 
    	&= \sum_{j=0}^{n-1} a_{j+q-n+2} \cdots a_q a_1 \cdots a_{j+q-n} \otimes a_{j+q-n+1},
        \label{(d_2)_q}\\
    (\tilde{d}_{3})_q(a_1 \cdots a_{q-n-1} \otimes a_{q-n} \cdots a_{q}) 
    	&= a_1 \cdots a_{q-n} \otimes a_{q-n+1} \cdots a_q \nonumber \\
        	& \quad\quad -  a_q a_1 \cdots a_{q-n-1} \otimes a_{q-n} \cdots a_{q-1}.\label{(d_3)_q}
\end{align}
Note that $(\tilde{d}_2)_q = 0$ for each $q\,(\geq n+1)$.
Since $(A \otimes_{A^e} \textrm{\boldmath{$P$}}_3)_n =0$,
we have $(\tilde{d}_3)_n = 0$.

Moreover,
$A \otimes_{A^e} \textrm{\boldmath{$P$}}$
is not only $\mathbb{N}$-graded, but also $\Delta_q^{\rm c}/C_q$-graded.
A path $\gamma$ in $\Delta$ is a {\it cycle} if $s(\gamma) = t(\gamma)$.
The set of cycles of length $q$ is denoted by $\Delta_q^{\rm c} (\subset  \Delta_q)$.
A cycle $\gamma$ is a {\it basic} cycle provided that we can not write
$\gamma = \beta^i$, for $i \geq 2$.
The set of basic cycles of length $q$ is denoted by $\Delta_q^{\rm b}$.
Then we have the following isomorphisms:
\begin{align*}
	(A \otimes_{A^e} \textrm{\boldmath {$P$}}_1)_q
    	&\cong \begin{cases}
        	K\Delta_q^{\rm c}  & \text{if $1 \leq q \leq n$}, \\
            0 & \text{otherwise},
        \end{cases} \\
	(A \otimes_{A^e} \textrm{\boldmath {$P$}}_2)_q
    	&\cong \begin{cases}
        	K\Delta_q^{\rm c}  &\text{if $n \leq q \leq 2n-1$}, \\
            0 & \text{otherwise},
        \end{cases} \\
	(A \otimes_{A^e} \textrm{\boldmath {$P$}}_3)_q
    	&\cong \begin{cases}
        	K\Delta_q^{\rm c}  &\text{if $n+1 \leq q \leq 2n$}, \\
            0 & \text{otherwise}.
        \end{cases} \\
\end{align*}
%
%
Let $C_q$ be the cyclic group of order $q$, with generator $c$.
Then we define an action of $C_q$ on $\Delta_q^{\rm c}$ by
$c(a_1 \cdots a_q) = a_q a_1 \cdots a_{q-1}$.
For each $\gamma \in \Delta_q^{\rm c}$,
we define the {\it orbit} of $\gamma$ to be the subset
$\overline{\gamma} = \{ c^i (\gamma)\, |\, 1 \leq i \leq q \} \in \Delta_q^{\rm c}$.
We denote the set of orbits by $\Delta_q^{\rm c}/C_q$.
Since the differential $(\overline{d_*})_q$ preserves $\overline{\gamma}$,
the complex
$
	(A 
    	\otimes_{A^e} 
        	\textrm{\boldmath {$P$}}_*)_q
$
splits into subcomplexes
by
$$
	(A \otimes_{A^e} \textrm{\boldmath {$P$}}_i)_q
    	\cong
        	\bigoplus_{\overline{\gamma} \in \Delta_q^{\rm c}/C_q} 
            	((A \otimes_{A^e} \textrm{\boldmath {$P$}}_i)_q)_{\overline{\gamma}},
$$
for
$\overline{\gamma} \in \Delta_q^{\rm c}/C_q$,
where 
$
((A \otimes_{A^e} \textrm{\boldmath {$P$}}_i)_q)_{\overline{\gamma}}
	= \bigoplus_{\gamma' \in \overline{\gamma}} K{\gamma'}
$
the direct sum of subspaces.
We express this fact by saying that
$((A \otimes_{A^e} \textrm{\boldmath{$P$}}_*)_q ,\, (\tilde{d}_*)_q)$ is
$\Delta_q^{\rm c}/C_q$-graded (cf. \cite{skoldberg}).
Sk\"oldberg computed the homology groups $HH_{p,q}(A)$ by exploiting
the above fact that 
$((A \otimes_{A^e} \textrm{\boldmath{$P$}}_*)_q ,\, (\tilde{d}_*)_q)$ is
$\Delta_q^{\rm c}/C_q$-graded.
We only show the second term of $HH_{p}(A)$.
\begin{theorem}[{See \cite[Theorem 2]{skoldberg}}]\label{2nd Hochschild}
    Let $A$ be a truncated quiver algebra $K\Delta/R_\Delta^n$.
    Then the degree $q$ part of the second Hochschild homology $HH_{2,\,q}(A)$ is given by
	\begin{align*}
    	HH_{2,\,q}&(A) \\
        & \hspace{-10pt}= 
        \begin{cases}
        	K^{a_q}  &  {\it if}\: n+1 \leq q \leq 2n-1 , \\
            \bigoplus_{r|q}
            	(K^{{\rm gcd}(n,r)-1} 
                	\oplus {\rm Ker} (\cdot \frac{n}{{\rm gcd}(n,r)}: K \rightarrow K
                ))^{b_r}
                & {\it if} \: q=n, \\
            0 & otherwise.
        \end{cases}
    \end{align*}
    Here we set $a_q := {\rm card}(\Delta_q^{\rm c}/C_q)$
    and $b_r := {\rm card}(\Delta_r^{\rm b}/C_r)$.
\end{theorem}
%
%
Next, Cibils gave the following another projective resolution of $A$
as a left $A^e$- module.
\begin{lemma}[{See \cite[Lemma 1.1]{Cibils}}]
    Let $A$ be a truncated quiver algebra $K\Delta/R_\Delta^n$
    and $J=J(A)$ the Jacobson radical of $A$.
    Then there exists the following projective resolution of $A$ as a left $A^e$-module\,{\rm :}
    \begin{align*}
    	\textrm{\boldmath $Q$}:
        	\cdots \stackrel{}{\longrightarrow}
        		A \otimes _{K\Delta_0} J^{\otimes_{K\Delta_0}^3} &\otimes_{K\Delta_0} A\stackrel{\partial_3}{\longrightarrow}
                	A \otimes _{K\Delta_0} J^{\otimes_{K\Delta_0}^2} \otimes_{K\Delta_0} A \\
                      &\stackrel{\partial_2}{\longrightarrow}
                    	A \otimes _{K\Delta_0} J \otimes_{K\Delta_0} A\stackrel{\partial_1}{\longrightarrow}
                        	A \otimes_{K\Delta_0} A\stackrel{\partial_0}{\longrightarrow}
                            	A \stackrel{}{\longrightarrow}
                                	0.
    \end{align*}
	Here the differentials $\partial_2$ and $\partial_3$ are defined by
    \begin{align*}
    	\partial_2(x \otimes y_1 \otimes y_2 \otimes z)
        	&= xy_1 \otimes y_2 \otimes z
            	- x \otimes y_1 y_2 \otimes z
                	+ x \otimes y_1 \otimes y_2 z, \\
        \partial_3(x \otimes y_1 \otimes y_2 \otimes y_3 \otimes z)
        	&= x y_1 \otimes y_2 \otimes y_3 \otimes z
            	- x \otimes y_1 y_2 \otimes y_3 \otimes z \\
                	&\quad\quad\quad+ x \otimes y_1 \otimes y_2 y_3 \otimes z
                    	- x \otimes y_1 \otimes y_2 \otimes y_3 z,
    \end{align*}
    for $x,\,z \in A$ and $y_i \in J\,(1 \leq i \leq 3)$.
\end{lemma}
In \cite{ACT}, Ames, Cagliero and Tirao gave
a chain map between the projective resolutions \textrm{\boldmath $Q$}
and \textrm{\boldmath $P$}.
We describe the second term of the chain map
$
	\pi : \textrm{\boldmath $Q$} \rightarrow \textrm{\boldmath $P$}
$
in Proposition \ref{chain cpx}.
\begin{proposition}[{See \cite[Section 4.2]{ACT}}]\label{chain cpx}
    Let $x_1$, $x_2$ be paths in $\Delta$.
    We set
    $
    	x_1 = a_1 a_2 \cdots a_{m_1},
    $
    $
    	x_2 = a_{m_1+1} a_{m_1+2} \cdots a_{m_1 + m_2},
    $
    where $a_1, a_2, \ldots , a_{m_1 + m_2} \in \Delta_1$.
    Then there exists a map
    $
		\pi_2 : \textrm{\boldmath $Q$}_2 \rightarrow \textrm{\boldmath $P$}_2
    $
    defined by the following equation{\rm :}
    \begin{align*}
    	\pi_2(a \otimes_{K\Delta_0} &x_1 \otimes_{K\Delta_0} x_2 \otimes_{K\Delta_0} b) \\
        	&= \begin{cases}
            	a \otimes_{K\Delta_0} a_1 \cdots a_n 
                	\otimes_{K\Delta_0} a_{n+1} \cdots a_{m_1 + m_2}b 
                &\text{if $m_1 + m_2 \geq n$,} \\
            	0 &{\it otherwise},
            \end{cases}
    \end{align*}
    for $a,\,b \in A$.
\end{proposition}

%
%
%
To provide an isomorphism
$D(HH_{2}(A)) \cong H^2(A,\,D(A))$,
we define a homomorphism $\Theta : D(A \otimes_{K\Delta_0^e} K\Delta_n) \rightarrow \Hom_K(A^{\otimes 2},\,D(A))$.
In order to define $\Theta$, we define some homomorphisms.
We define the canonical isomorphism 
$
	\varphi : 
    	A \otimes_{A^e} \textrm{\boldmath $P$}_2 
    		\rightarrow
        		A \otimes_{K\Delta_0^e} K\Delta_n
$
in a similar way as (\ref{iso}),
$
	\nu_2 : A^{\otimes 4}
			\rightarrow
            	\textrm{\boldmath $Q$}_2
$
by
$$
	1 \otimes_K a_1 \otimes_K a_2 \otimes_K 1
    	\mapsto
        	\begin{cases}
            	1 \otimes_{K\Delta_0} a_1 \otimes_{K\Delta_0} a_2 \otimes_{K\Delta_0} 1 & {\rm if}\: a_1,\,a_2 \in J, \\
                0 & {\rm otherwise},
    		\end{cases}
$$
and
$\mu_2$ by (\ref{mu}).
We denote by $\Theta$ the composition of following maps
\begin{align*}
	D(A \otimes_{K\Delta_0^e} K\Delta_n)
    	&\xrightarrow{D\varphi} 
        	D(A \otimes_{A^e} \textrm{\boldmath $P$}_2) \\
            	&\xrightarrow{D({\rm id} \otimes \pi_2)} 
                	D(A \otimes_{A^e}\textrm{\boldmath $Q$}_2) \\
                    	&\xrightarrow{D({\rm id} \otimes \nu_2)} 
                        	D(A \otimes_{A^e} A^{\otimes 4}) \\
                            	&\xrightarrow{{\rm adj}} 
                                	\Hom_{A^e}(A^{\otimes 4},\,D(A)) \\
                                    	&\xrightarrow{\mu_2}
                                        	\Hom_K(A^{\otimes 2},\,D(A)),
\end{align*}
where 
$
	{\rm adj} 
    	: D(A \otimes_{A^e} A^{\otimes 4})
        	\rightarrow 
        		\Hom_{A^e}(A^{\otimes 4},\,D(A))
$
is the isomorphism induced by a natural isomorphism
$
D(A \otimes_{A^e} -)
	\rightarrow 
    	\Hom_{A^e}(-,\,D(A))
$,
namely,
for $g \in D(A \otimes_{A^e} A^{\otimes 4})$
and $x \in A^{\otimes 4}$,
$({\rm adj}(g))(x)$ is given by sending
$a \in A$ to
$g(a \otimes_{A^e} x)$.
For
$
x \otimes_{A^e} (1 \otimes_{K} a_1 \cdots a_{m_1} \otimes_{K} a_{m_1+1} \cdots a_{m_1+m_2} \otimes_K 1)
	\in A \otimes_{A^e} A^{\otimes 4}\,
    	(a_i \in \Delta_1),
$
we have
\begin{align*}
	x \otimes_{A^e} &(1 \otimes_{K} a_1 \cdots a_{m_1} \otimes_{K} a_{m_1+1} \cdots a_{m_1+m_2} \otimes_K 1) \\
    	&\xmapsto{{\rm id} \otimes \nu_2}
        	x \otimes_{A^e} (1 \otimes_{K\Delta_0} a_1 \cdots a_{m_1} \otimes_{K\Delta_0} a_{m_1+1} \cdots a_{m_1+m_2} \otimes_{K\Delta_0} 1) \\
    			&\xmapsto{{\rm id} \otimes\pi_2}
                	\begin{cases}
                    	x \otimes_{A^e} (1 \otimes_{K\Delta_0} a_1 \cdots a_{n} \otimes_{K\Delta_0} a_{n+1} \cdots a_{m_1+m_2})
                			& {\rm if}\: n \leq m_1 + m_2, \\
                        0
                        	& {\rm otherwise},
                    \end{cases} \\
                		&\xmapsto{\varphi}
                        	\begin{cases}
                    			a_{n+1} \cdots a_{m_1+m_2}x \otimes_{A^e} a_1 \cdots a_{n})
                					& {\rm if}\: n \leq m_1 + m_2, \\
                        		0
                        			& {\rm otherwise}.
                    		\end{cases}
\end{align*}
Hence, for an dual basis element
$
	u^* 
    	:= (a_{n+1} \cdots a_r \otimes_{K\Delta_0^e} a_1 a_2 \cdots a_n)^* 
        	\in D(A \otimes_{K\Delta_0^e} K\Delta_n)
    \,(a_i \in \Delta_1)
$,
$\Theta(u^*) \in \Hom_{K}(A^{\otimes 2},\, D(A))$ is the map as follows:
\begin{align}
	b_1 \cdots b_{m_1} &\otimes_K b_{m_1 +1} \cdots b_{m_1 + m_2} \notag \\
    	&\mapsto
            	\begin{cases}
                	(a_{m_1 +m_2 +1} \cdots a_r)^* 
                    	& \text{if $\: n \leq m_1 + m_2 \leq r$} \\
                    	&\text{and $b_t = a_t$ for $t (1 \leq t \leq m_1 + m_2)$,} \\
                    0 
                    	&{\rm otherwise}.\\
                \end{cases} \label{Theta}
\end{align}
The isomorphism
$D(HH_{2}(A)) \cong H^2(A,\,D(A))$
is induced by $\Theta$.
In Section \ref{main theorem},
we get 2-cocycles from $D(HH_{2,\,q}(A))$ through
the following isomorphism:
$$
        	\bigoplus_{q} D(HH_{2,\,q}(A))
            	\cong D(\bigoplus_{q} HH_{2,\,q}(A))
                	= D(HH_2(A))
                    	\xrightarrow{\sim} {H^2(A,\,D(A))}.
$$
We denote the composition of the above isomorphisms by $\Theta$ again.
%
%
\section{The ordinary quivers of Hochschild extension algebras for self-injective Nakayama algebras}\label{main theorem}
In this section, we compute the ordinary quivers of Hochschild extension algebras
for self-injective Nakayama algebras.
In the following,
for an algebra $A$, we denote the Jacobson radical of $A$ by $J(A)$.

Let $\Delta$ be a finite quiver and $A=K\Delta/I$ for an admissible ideal $I$.
Let $\Delta_0 = \{1, 2, \ldots, l \}$ be the set of vertices of $\Delta$.
For each pair of integers $i$ and $j$
with $i,\,j \in \Delta_0$,
we denote
the number of arrows from $i$ to $j$ in the ordinary quiver of $A$
by ${\rm N}_A (i,\,j)$.
%
%
\begin{lemma}\label{key lemma2}
	Let $\Delta$ be a finite quiver and $A=K\Delta/I$ for an admissible ideal $I$.
    Let $T_\alpha(A)$ be an extension algebra of $A$ defined by a $2$-cocycle
    $\alpha : A \times A \rightarrow D(A)$.
    We denote by $\Delta_{T_\alpha(A)}$ and $\Delta_{T_0(A)}$ the ordinary quiver of ${T_\alpha(A)}$
    and the trivial extension algebra ${T_0(A)}$, respectively.
    If $\alpha(e_i,\,-) = \alpha(-,\,e_i) = 0$ for all $i \in \Delta_0$,
    then we have the chain of subquivers of $\Delta_{T_0}(A)${\rm :}
    $$
    	\Delta \subseteq  \Delta_{T_\alpha(A)} \subseteq \Delta_{T_0(A)}.
    $$
\end{lemma}
\begin{proof}
Let $\Delta_0 = \{1, 2, \ldots, l \}$ be the set of vertices of $\Delta$
and let $\{e_1, \ldots , e_l \}$ be the set of trivial paths in $K\Delta$.
Since $\alpha(e_i,\,-) = \alpha(-,\,e_i) = 0$ for all $i \in \Delta_0$,
it is easy to verify that
$
	\{(e_1,\,0), \ldots, (e_l,\,0)\}
$
is a complete set of primitive orthogonal idempotents in $T_\alpha(A)$.
So, $\Delta_{T_\alpha(A)}$ has $n$ vertices in a $1$-$1$-correspondence with
$
	(e_1,\,0), \ldots, (e_l,\,0).
$
For each pair of integers $i$ and $j$
with $i,\,j \in \Delta_0$,
we obtain
\begin{align*}
	{\rm N}_{T_\alpha(A)}&(i,\,j) \\
     &={\rm dim}_K
    	\frac{(e_i,\,0)J(T_\alpha(A))(e_j,\,0)}
             {(e_i,\,0)J^2(T_\alpha(A))(e_j,\,0)} \\
        & = {\rm dim}_K
        	\frac{(e_i J(A) e_j ,\, e_i D(A) e_j)}
                 {(e_i J^2(A)e_j ,\, e_i(J(A)D(A) + D(A) J(A) + \alpha(J(A),\,J(A))) e_j)} \\
                & ={\rm dim}_K 
                 	\frac{e_i J(A) e_j}
                         {e_i J^2(A)e_j}
                  + {\rm dim}_K
                     \frac{e_i D(A) e_j}
                     	  {e_i(J(A)D(A) + D(A) J(A) + \alpha(J(A),\,J(A))) e_j}\\
                & = {\rm dim}_K 
                 	\frac{e_i J(A) e_j}
                         {e_i J^2(A)e_j}
                 + {\rm dim}_K e_i D(A) e_j \\
                &\quad\quad\quad\quad 
                 - {\rm dim}_K\, e_i(J(A)D(A) + D(A) J(A) + \alpha(J(A),\,J(A))) e_j.
\end{align*}
Since
$$
	{\rm N}_{A}(i,\,j) 
    	= {\rm dim}_K \frac{e_i J(A) e_j}
                           {e_i J^2(A)e_j},
$$
we have ${\rm N}_{A}(i,\,j) \leq {\rm N}_{T_\alpha(A)}(i,\,j)$.
Consequently, $\Delta \subseteq  \Delta_{T_\alpha(A)}$.

On the other hand, by \cite[Proposition 2.2]{Fernandez}, we have
\begin{align*}
	{\rm N}_{T_0(A)}(i,\,j)
    	& = {\rm dim}_K 
                 	\frac{e_i J(A) e_j}
                         {e_i J^2(A)e_j}
          + {\rm dim}_K
          			\frac{e_i D(A) e_j}
                    	 {e_i(J(A)D(A) + D(A) J(A)) e_j} \\
		& = {\rm dim}_K 
                 	\frac{e_i J(A) e_j}
                         {e_i J^2(A)e_j}
          + {\rm dim}_K e_i D(A) e_j \\
        &\quad\quad
          - {\rm dim}_K\, e_i(J(A)D(A) + D(A) J(A)) e_j.
\end{align*}
Since
$$
	e_i(J(A)D(A) + D(A) J(A)) e_j
		\subseteq 
     		e_i(J(A)D(A) + D(A) J(A) + \alpha(J(A),\,J(A))) e_j,
$$
we have ${\rm N}_ {T_\alpha(A)}(i,\,j) \leq {\rm N}_ {T_0(A)}(i,\,j)$.
Hence $\Delta_{T_\alpha(A)}$ is a subquiver of $\Delta_{T_0(A)}$.
\end{proof}
%
%
\begin{lemma}\label{key lemma}
	Let $\Delta$ be a finite quiver and $A=K\Delta/I$ for an admissible ideal $I$.
    Let $T_\alpha(A)$ be an extension algebra of $A$ defined by a $2$-cocycle
    $\alpha : A \times A \rightarrow D(A)$.
    If $\alpha(e_i,\,-) = \alpha(-,\,e_i) = 0$ for all $i \in \Delta_0$,
    then the following conditions are equivalent{\rm :}
    \begin{itemize}
 		\item[\rm{(1)}] $\alpha(J(A),\,J(A)) \subseteq J(A)D(A) + D(A) J(A)$.
 		\item[\rm{(2)}] $\Delta_{T_\alpha(A)} = \Delta_{T_0(A)}$.
	\end{itemize}
\end{lemma}
\begin{proof}
It is clear that the vertices $(\Delta_{T_\alpha(A)})_0$ and $(\Delta_{T_0(A)})_0$
coincide with $\Delta_0$.
If $\alpha(J(A),\,J(A)) \subseteq J(A)D(A) + D(A) J(A)$,
then
for each pair of integers $i$ and $j$
with $1 \leq i,\,j \leq l$,
we have
\begin{align*}
	{\rm N}_{T_\alpha(A)}&(i,\,j) \\
    	& ={\rm dim}_K 
                 	\frac{e_i J(A) e_j}
                         {e_i J^2(A)e_j}
                  + {\rm dim}_K
                     \frac{e_i D(A) e_j}
                     	  {e_i(J(A)D(A) + D(A) J(A) + \alpha(J(A),\,J(A))) e_j} \\
        	& = {\rm dim}_K
                 	\frac{e_i J(A) e_j}
                         {e_i J^2(A)e_j}
          		+ {\rm dim}_K
          			\frac{e_i D(A) e_j}
                         {e_i(J(A)D(A)+ D(A) J(A)) e_j}.
\end{align*}
Hence ${\rm N}_{T_\alpha(A)}(i,\,j)$ is equal to ${\rm N}_{T_0(A)}(i,\,j)$.

Conversely, if ($1$) does not hold, then there exist $x \in \alpha(J(A),\,J(A))$
and $i,\,j \in \Delta_0$ such that
$x \notin J(A)D(A) + D(A) J(A)$
and $e_i x e_j \neq 0$ in $D(A)$.
Therefore, we have
\begin{align*}
	{\rm N}_{T_0(A)}&(i,\,j) - {\rm N}_{T_\alpha(A)}(i,\,j) \\
    	&= {\rm dim}_K {e_i(J(A)D(A) + D(A) J(A) + \alpha(J(A),\,J(A))) e_j} \\
        	&\qquad\qquad\qquad - {\rm dim}_K {e_i(J(A)D(A)+ D(A) J(A)) e_j} \\
        &= {\rm dim}_K e_i \frac{J(A)D(A) + D(A) J(A) + \alpha(J(A),\,J(A))}
        						{J(A)D(A)+ D(A) J(A)} e_j 
        > 0.
\end{align*}
Hence $\Delta_{T_\alpha(A)}$ is not equal to $\Delta_{T_0(A)}$.
\end{proof}

%
%
From now on, we consider self-injective Nakayama algebras.
Let $\Delta$ be the following cyclic quiver
with $s\,(\geq 1)$ vertices and $s$ arrows:
$$
{\unitlength 0.1in
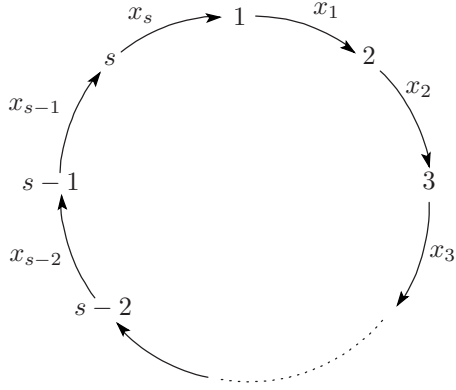
\begin{picture}( 23.4600, 20.2400)( 22.0400,-34.4700)
%
{\color[named]{Black}{%
\special{pn 8}%
\special{pa 2938 1720}%
\special{pa 2964 1700}%
\special{pa 3016 1664}%
\special{pa 3072 1632}%
\special{pa 3100 1618}%
\special{pa 3130 1604}%
\special{pa 3160 1592}%
\special{pa 3188 1580}%
\special{pa 3220 1570}%
\special{pa 3280 1554}%
\special{pa 3344 1542}%
\special{pa 3374 1536}%
\special{pa 3406 1534}%
\special{pa 3418 1534}%
\special{fp}%
}}%
%
{\color[named]{Black}{%
\special{pn 8}%
\special{pa 3418 1540}%
\special{pa 3454 1538}%
\special{fp}%
\special{sh 1}%
\special{pa 3454 1538}%
\special{pa 3386 1522}%
\special{pa 3402 1542}%
\special{pa 3390 1562}%
\special{pa 3454 1538}%
\special{fp}%
}}%
\put(35.5900,-15.3200){\makebox(0,0){$1$}}%
%
{\color[named]{Black}{%
\special{pn 8}%
\special{pa 3642 1532}%
\special{pa 3706 1536}%
\special{pa 3770 1544}%
\special{pa 3800 1548}%
\special{pa 3832 1556}%
\special{pa 3862 1564}%
\special{pa 3894 1572}%
\special{pa 3954 1596}%
\special{pa 4012 1622}%
\special{pa 4068 1652}%
\special{pa 4094 1668}%
\special{pa 4122 1686}%
\special{pa 4130 1692}%
\special{fp}%
}}%
%
{\color[named]{Black}{%
\special{pn 8}%
\special{pa 4126 1698}%
\special{pa 4158 1722}%
\special{fp}%
\special{sh 1}%
\special{pa 4158 1722}%
\special{pa 4116 1666}%
\special{pa 4114 1690}%
\special{pa 4092 1698}%
\special{pa 4158 1722}%
\special{fp}%
}}%
%
{\color[named]{Black}{%
\special{pn 8}%
\special{pa 4294 1820}%
\special{pa 4318 1842}%
\special{pa 4360 1890}%
\special{pa 4380 1914}%
\special{pa 4416 1966}%
\special{pa 4432 1994}%
\special{pa 4450 2022}%
\special{pa 4464 2050}%
\special{pa 4492 2108}%
\special{pa 4504 2136}%
\special{pa 4516 2168}%
\special{pa 4536 2228}%
\special{pa 4544 2258}%
\special{pa 4548 2280}%
\special{fp}%
}}%
%
{\color[named]{Black}{%
\special{pn 8}%
\special{pa 4540 2282}%
\special{pa 4550 2320}%
\special{fp}%
\special{sh 1}%
\special{pa 4550 2320}%
\special{pa 4554 2250}%
\special{pa 4538 2268}%
\special{pa 4514 2260}%
\special{pa 4550 2320}%
\special{fp}%
}}%
%
{\color[named]{Black}{%
\special{pn 8}%
\special{pa 4550 2508}%
\special{pa 4550 2572}%
\special{pa 4548 2604}%
\special{pa 4544 2636}%
\special{pa 4540 2666}%
\special{pa 4534 2698}%
\special{pa 4520 2760}%
\special{pa 4512 2792}%
\special{pa 4482 2882}%
\special{pa 4468 2912}%
\special{pa 4440 2968}%
\special{pa 4424 2996}%
\special{pa 4412 3018}%
\special{fp}%
}}%
%
{\color[named]{Black}{%
\special{pn 8}%
\special{pa 4406 3014}%
\special{pa 4384 3046}%
\special{fp}%
\special{sh 1}%
\special{pa 4384 3046}%
\special{pa 4438 3002}%
\special{pa 4414 3002}%
\special{pa 4404 2980}%
\special{pa 4384 3046}%
\special{fp}%
}}%
\put(42.3600,-17.4100){\makebox(0,0){$2$}}%
\put(45.5000,-23.9400){\makebox(0,0){$3$}}%
%
{\color[named]{Black}{%
\special{pn 8}%
\special{pn 8}%
\special{pa 4306 3126}%
\special{pa 4300 3132}%
\special{fp}%
\special{pa 4276 3160}%
\special{pa 4270 3166}%
\special{fp}%
\special{pa 4244 3192}%
\special{pa 4238 3198}%
\special{fp}%
\special{pa 4210 3222}%
\special{pa 4204 3226}%
\special{fp}%
\special{pa 4176 3250}%
\special{pa 4170 3254}%
\special{fp}%
\special{pa 4140 3276}%
\special{pa 4134 3282}%
\special{fp}%
\special{pa 4104 3302}%
\special{pa 4098 3306}%
\special{fp}%
\special{pa 4066 3324}%
\special{pa 4060 3328}%
\special{fp}%
\special{pa 4028 3344}%
\special{pa 4022 3348}%
\special{fp}%
\special{pa 3990 3364}%
\special{pa 3983 3367}%
\special{fp}%
\special{pa 3950 3380}%
\special{pa 3943 3384}%
\special{fp}%
\special{pa 3909 3396}%
\special{pa 3902 3400}%
\special{fp}%
\special{pa 3868 3410}%
\special{pa 3861 3412}%
\special{fp}%
\special{pa 3826 3421}%
\special{pa 3818 3424}%
\special{fp}%
\special{pa 3783 3431}%
\special{pa 3776 3432}%
\special{fp}%
\special{pa 3739 3438}%
\special{pa 3732 3440}%
\special{fp}%
\special{pa 3695 3444}%
\special{pa 3687 3444}%
\special{fp}%
\special{pa 3650 3446}%
\special{pa 3642 3446}%
\special{fp}%
\special{pa 3604 3448}%
\special{pa 3596 3448}%
\special{fp}%
\special{pa 3559 3446}%
\special{pa 3551 3446}%
\special{fp}%
\special{pa 3514 3442}%
\special{pa 3506 3442}%
\special{fp}%
\special{pa 3470 3436}%
\special{pa 3462 3436}%
\special{fp}%
}}%
%
{\color[named]{Black}{%
\special{pn 8}%
\special{pa 3392 3424}%
\special{pa 3362 3418}%
\special{pa 3330 3410}%
\special{pa 3270 3390}%
\special{pa 3240 3378}%
\special{pa 3212 3366}%
\special{pa 3154 3338}%
\special{pa 3126 3322}%
\special{pa 3100 3306}%
\special{pa 3072 3288}%
\special{pa 3046 3270}%
\special{pa 2996 3230}%
\special{pa 2972 3208}%
\special{pa 2950 3186}%
\special{pa 2938 3176}%
\special{fp}%
}}%
%
{\color[named]{Black}{%
\special{pn 8}%
\special{pa 2944 3170}%
\special{pa 2918 3142}%
\special{fp}%
\special{sh 1}%
\special{pa 2918 3142}%
\special{pa 2948 3204}%
\special{pa 2954 3182}%
\special{pa 2978 3178}%
\special{pa 2918 3142}%
\special{fp}%
}}%
%
{\color[named]{Black}{%
\special{pn 8}%
\special{pa 2618 2352}%
\special{pa 2622 2288}%
\special{pa 2628 2258}%
\special{pa 2634 2226}%
\special{pa 2642 2194}%
\special{pa 2658 2134}%
\special{pa 2668 2102}%
\special{pa 2680 2074}%
\special{pa 2690 2044}%
\special{pa 2704 2014}%
\special{pa 2718 1986}%
\special{pa 2734 1958}%
\special{pa 2748 1930}%
\special{pa 2766 1902}%
\special{pa 2784 1876}%
\special{pa 2798 1856}%
\special{fp}%
}}%
%
{\color[named]{Black}{%
\special{pn 8}%
\special{pa 2804 1860}%
\special{pa 2826 1830}%
\special{fp}%
\special{sh 1}%
\special{pa 2826 1830}%
\special{pa 2770 1872}%
\special{pa 2794 1872}%
\special{pa 2804 1896}%
\special{pa 2826 1830}%
\special{fp}%
}}%
%
{\color[named]{Black}{%
\special{pn 8}%
\special{pa 2800 3008}%
\special{pa 2780 2982}%
\special{pa 2764 2956}%
\special{pa 2748 2928}%
\special{pa 2730 2900}%
\special{pa 2702 2844}%
\special{pa 2678 2784}%
\special{pa 2658 2724}%
\special{pa 2650 2694}%
\special{pa 2636 2632}%
\special{pa 2628 2600}%
\special{pa 2624 2568}%
\special{pa 2622 2536}%
\special{pa 2620 2512}%
\special{fp}%
}}%
%
{\color[named]{Black}{%
\special{pn 8}%
\special{pa 2628 2512}%
\special{pa 2626 2474}%
\special{fp}%
\special{sh 1}%
\special{pa 2626 2474}%
\special{pa 2610 2542}%
\special{pa 2628 2528}%
\special{pa 2650 2540}%
\special{pa 2626 2474}%
\special{fp}%
}}%
\put(28.8000,-17.5700){\makebox(0,0){$s$}}%
\put(25.7000,-24.0300){\makebox(0,0){$s-1$}}%
\put(28.4300,-30.5800){\makebox(0,0){$s-2$}}%
\put(40.1400,-15.0300){\makebox(0,0){$x_1$}}%
\put(44.9800,-19.1600){\makebox(0,0){$x_2$}}%
\put(46.2300,-27.7700){\makebox(0,0){$x_3$}}%
\put(30.4200,-15.2800){\makebox(0,0){$x_s$}}%
\put(24.8900,-20.0900){\makebox(0,0){$x_{s-1}$}}%
\put(24.9300,-28.0000){\makebox(0,0){$x_{s-2}$}}%
\end{picture}}%

$$
Suppose $n \geq 2$ and $A = K\Delta/R_\Delta^n$,
which is called a {\it truncated cycle algebra} in \cite{Bardzell},
where $R_\Delta^n$ is the two-sided ideal of $K\Delta$
generated by the paths of length $n$.
We regard the subscripts $i$ of $e_i$ and $x_i$ modulo $s\, (1 \leq i \leq s)$.
By Theorem \ref{2nd Hochschild},
the second Hochschild homology is given by
\begin{equation}\label{2nd Hochschild of truncated}
    	HH_{2,\,q}(A) = 
        \begin{cases}
        	K  &  \text{if $s|q$ and $n+1 \leq q \leq 2n-1$}, \\
            K^{s-1} 
            	\oplus
                	{\rm Ker} (\cdot \frac{n}{s}: K \rightarrow K)
                & \text{if $s|q$ and $q=n$}, \\
            0 & \text{otherwise}.
        \end{cases}
\end{equation}
We have the following main theorem about the ordinary quiver of Hochschild extension algebras.
\begin{theorem}\label{Main theorem}
	Suppose that $n \geq 2$, $A = K\Delta/R_\Delta^n$ and
	$n \leq q \leq 2n-1$.
    Let $\alpha : A \times A \rightarrow D(A)$
    be a $2$-cocycle such that the cohomology class
    $[\alpha]$ of $\alpha$
    belongs to $\Theta(D(HH_{2,\,q}(A)))$,
    and let $T_\alpha(A)$ be the Hochschild extension algebra of $A$
    defined by $\alpha$.
	Then the ordinary quiver $\Delta_{T_\alpha(A)}$ is given by
	$$
		\Delta_{T_\alpha(A)}
    		= \begin{cases}
        		\Delta_{T_0(A)} & {\it if}\quad n \leq q \leq 2n-2, \\
            	\Delta          & {\it if}\quad q = 2n-1.
        	\end{cases}
	$$
\end{theorem}
\begin{proof}
If $s \nmid q$, then $HH_{2,\,q}(A) = 0$ by (\ref{2nd Hochschild of truncated}).
It is suffices to consider the case $q$ be divided by $s$.
We will investigate the ordinary quiver
dividing the proof into the following three cases:
\begin{itemize}
	 \setlength{\leftskip}{12pt}
	\item[{\it Case} $1$:] $n+1 \leq q \leq 2n-2$, 
	\item[{\it Case} $2$:] $q = n$,
	\item[{\it Case} $3$:] $q = 2n-1$. 
\end{itemize}

{\it Case} $1$:
We set
$$
	v_i = x_{i+n} \cdots x_{i+q-1} \otimes_{K\Delta_0^e} x_i x_{i+1} \cdots x_{i+n-1},
$$
and
$$
	w_i = x_{i+n+1} \cdots x_{i+q-1} \otimes_{K\Delta_0^e} x_i x_{i+1} \cdots x_{i+n},
$$
for $1 \leq i \leq s$.
Then
$\{v_1 , \ldots , v_s \}$ is the basis of $(A \otimes_{K\Delta_0^e} K\Delta_{n})_q$,
and
$\{w_1 , \ldots , w_s \}$ is the basis of $(A \otimes_{K\Delta_0^e} K\Delta_{n+1})_{q}$.
By (\ref{(d_3)_q}),
$
	(\tilde{d}_{3})_q
    	: (A \otimes_{K\Delta_0^e} K\Delta_{n+1})_{q} 
        	\rightarrow 
            	(A \otimes_{K\Delta_0^e} K\Delta_{n})_q
$
is given by
\begin{equation}\label{d_3}
	(\tilde{d}_{3})_q (w_1, \ldots , w_s)
    	=(v_1 , \ldots , v_s)
        	\left(
				\begin{array}{ccccc}
 					-1     & 0      & \cdots & 0      & 1 \\
 					 1     &-1      & \ddots &        & 0 \\
 					 0     & \ddots & \ddots & \ddots & \vdots \\
 					 \vdots&        & \ddots & \ddots & 0 \\
 					 0     & \cdots & 0      & 1      & -1
				\end{array}
			\right)
\end{equation}
and by (\ref{(d_2)_q}), $(\tilde{d}_{2})_q$ is the zero map.
By (\ref{d_3}) and the fact $(\tilde{d}_2)_q = 0$,
we have the following isomorphism:
\begin{align}
	D(HH_{2,\,q}(A))
    	&\cong {\rm Ker}\,(D((\tilde{d}_{3})_q))/{\rm Im}\,(D((\tilde{d}_2)_q)) \notag \\
        	&= {\rm Ker}\,(D((\tilde{d}_{3})_q))
            = \langle v_1^* + \cdots + v_s^* \rangle, \label{case1}
\end{align}
where $v_i^* \in D((A \otimes _{K\Delta_0^e} \Delta_n)_q)$
is the dual basis element for each $1 \leq i \leq s$.
For $1 \leq i \leq s$,
by (\ref{Theta}),
$\Theta(v_i^*)$ is the map as follows:
\begin{align*}
	a_1 \cdots a_{m_1} \otimes_{K}& a_{m_1+1} \cdots a_{m_1 + m_2} \\
    	&\mapsto
         \begin{cases}
        	(x_{i+m_1+m_2} \cdots x_{i+q-1})^*
            &\text{if $n \leq m_1+m_2 \leq q$} \\ 
            &\text{and $a_t = x_{i+t-1}$ for $1 \leq t \leq m_1+m_2$}, \\
             0  &\text{otherwise}.
        \end{cases}
\end{align*}
We define a map
$\alpha_i : A \times A \rightarrow D(A)$ by 
$$
	\alpha_i(a_1 \cdots a_{m_1} ,\,a_{m_1+1} \cdots a_{m_1 + m_2}) 
    	= \Theta(v_i^*)(a_1 \cdots a_{m_1} \otimes_K a_{m_1+1} \cdots a_{m_1 + m_2}).
$$
For any $2$-cocycle
$\alpha : A \times A \rightarrow D(A)$,
there exists $k\, (\in K)$ such that
$[\alpha] = k[\sum_{i=1}^{s} \alpha_i]$.
Since, for $1 \leq i \leq s$,
we have
\begin{align*}
	(x_{i+m_1+m_2} \cdots x_{i+q-1})^* 
    	&= (x_{i+m_1+m_2-1} x_{i+m_1+m_2} \cdots x_{i+q-1})^* x_{i+m_1+m_2-1} \\
        	&\in J(A)D(A) + D(A) J(A),
\end{align*}
it follows that $\alpha$ satisfies the conditions of Lemma \ref{key lemma}.
Hence $\Delta_{T_\alpha(A)}$ coincides with $\Delta_{T_0(A)}$.

{\it Case} 2:
We consider the case $q=n$.
We set
$$
	u_i = x_{i+1} \cdots x_{i+n-1} \otimes_{K\Delta_0^e} x_i,
$$
and
$$
	v_i = 1 \otimes_{K\Delta_0^e} x_i x_{i+1} \cdots x_{i+n-1}
    		= e_i \otimes_{K\Delta_0^e} x_i x_{i+1} \cdots x_{i+n-1},
$$
for $1 \leq i \leq s$.
Then
$\{u_1 , \ldots , u_s \}$ is a basis of $(A \otimes_{K\Delta_0^e} K\Delta_1)_n$,
and
$\{v_1 , \ldots , v_s \}$ is a basis of $(A \otimes_{K\Delta_0^e} K\Delta_n)_n$.
By (\ref{(d_3)_q}), $(\tilde{d}_{3})_n$ is the zero map
and
$
	(\tilde{d}_{2})_n
    	: (A \otimes_{K\Delta_0^e} K\Delta_1)_n 
        	\rightarrow 
            	(A \otimes_{K\Delta_0^e} K\Delta_n)_n
$
is given by
$$
	(\tilde{d}_{2})_n (u_1, \ldots , u_s)
    	=(v_1 , \ldots , v_s)
        	\left(
				\begin{array}{ccccc}
 					n/s   & \cdots      & n/s  \\
 					\vdots       & \ddots      & \vdots       \\
 					n/s   & \cdots      & n/s
				\end{array}
			\right).
$$
If ${\rm char}(K) \nmid (n/s)$,
then we have
\begin{align}
	D(HH_{2,\,n}(A))
    	&\cong {\rm Ker}\,(D((\tilde{d}_3)_n))/{\rm Im}\,(D((\tilde{d}_2)_n)) \notag \\
        	&= \frac{\langle v_1^* ,\, \ldots ,\, v_s^* \rangle}{\langle v_1^* + \cdots + v_s^* \rangle}
            \cong \langle v_1^* ,\, \ldots ,\, v_{s-1}^* \rangle. \label{case2}
\end{align}
For $1 \leq i \leq s-1$,
$\Theta(v_i^*)$ is the map as follows:
$$
	a_1 \cdots a_{m_1} \otimes_K a_{m_1+1} \cdots a_{m_1 + m_2}
    	\mapsto
         \begin{cases}
        	e_i^* &\text{if $m_1+m_2 = n$} \\
            	  &\quad \text{and $a_t = x_{i+t-1}$ for $1 \leq t \leq n$,}\\
             0    &\text{otherwise}.
        \end{cases}
$$
Hence, we have a $2$-cocycle $\alpha_i : A \times A \rightarrow D(A)$ corresponding to
$v_i^*$ ($1 \leq i \leq s-1$) is given by
$$
 \alpha_i(a_1 \cdots a_{m_1} ,\, a_{m_1+1} \cdots a_{m_1 + m_2})
 	= \Theta(v_i^*)(a_1 \cdots a_{m_1} \otimes_K a_{m_1+1} \cdots a_{m_1 + m_2}).
$$
For any $2$-cocycle $\alpha$,
there exist $k_i \in K\,(1 \leq i \leq s-1)$ such that
$[\alpha] = \sum_{i=1}^{s-1} k_i [\alpha_i]$.
Since, for $1 \leq i \leq s-1$, we have
$$
	e_i^* 
    	= x_i x_i^*
        	\in J(A)D(A) + D(A) J(A),
$$
$\alpha$ satisfies the conditions of Lemma \ref{key lemma}.
Hence $\Delta_{T_{\alpha}(A)}$ coincides with $\Delta_{T_0(A)}$.

If ${\rm char}(K) \mid (n/s)$,
then we have the dual of the second Hochschild homology as follows:
\begin{align*}
	D(HH_{2,\,n}(A))
    	\cong {\rm Ker}\,(D(\tilde{d}_3)_n)/{\rm Im}\,(D(\tilde{d}_2)_n) 
        	= {\rm Ker}\,(D(\tilde{d}_3)_n)
        		= \langle v_1^* ,\, \ldots ,\, v_s^* \rangle.
\end{align*}
For any $2$-cocycle $\alpha$, there exist
$k_i \in K$ and $2$-cocycles $\alpha_i$ corresponding to $v_i^*$
$(1 \leq i \leq s)$ such that
$[\alpha] = \sum_{i=1}^{s} k_i [\alpha_i]$.
Then we have $\Delta_{T_{\alpha}(A)} = \Delta_{T_0(A)}$.
The proof is same as above.

{\it Case} $3$:
In this case we set $q=2n-1$.
We set
$$
	v_i = x_{i+n} \cdots x_{i+2n-2} \otimes_{K\Delta_0^e} x_i x_{i+1} \cdots x_{i+n-1}
$$
and
$$
	w_i = x_{i+n+1} \cdots x_{i+2n-2} \otimes_{K\Delta_0^e} x_i x_{i+1} \cdots x_{i+n},
$$
for $1 \leq i \leq s$.
In the same way to {\it Case} $1$,
we have
\begin{align}
	D(HH_{2,\,2n-1}(A))
    	&\cong {\rm Ker}\,(D((\tilde{d}_{3})_{2n-1}))/{\rm Im}\,(D((\tilde{d}_2)_{2n-1}))  \notag\\
        	&= {\rm Ker}\,(D((\tilde{d}_{3})_{2n-1})) 
            = \langle v_1^* + \cdots + v_s^* \rangle. \label{case3}
\end{align}
Then, for any $2$-cocycle $\alpha : A \times A \rightarrow D(A)$,
there exists $k\, (\in K)$ such that
$[\alpha] = k[\sum_{i=1}^{s} \alpha_i]$,
where
$$
	\alpha_i(a_1 \cdots a_{m_1} ,\, a_{m_1+1} \cdots a_{m_1 + m_2})
    	= \begin{cases}
        	(x_{m_1+m_2+i} \cdots x_{i+2n-2})^* \\
            \quad\quad \text{if $n \leq m_1+m_2 < 2n-1$}\\
            	\quad\quad \text{and $a_t = x_{i+t-1}$ for $1 \leq t \leq m_1+m_2$,} \\
             0 \quad \text{otherwise}.
        \end{cases}
$$
We will prove that $\Delta_{T_{\alpha}(A)}$ coincides with $\Delta$.
By \cite[Proposition 2.2]{Fernandez},
the quiver $\Delta_{T_0(A)}$ is given by
\begin{itemize}
 		\item[\rm{(1)}] $(\Delta_{T_0(A)})_0 = \Delta_0$,
        \item[\rm{(2)}] $(\Delta_{T_0(A)})_1 = \Delta_1 \cup \{y_1 ,\, \ldots , y_s \}$,
\end{itemize}
where $y_i$ is an arrow from $i$ to $i-n+1$, for $1 \leq i \leq s$.
Because of Lemma \ref{key lemma2},
it is sufficient to show that there is no arrow from $i$ to $i-n+1$ in $\Delta_{T_{\alpha}(A)}$.
Since 
$
	{\rm dim}_K\, {e_i J(A) e_{i-n+1}}/{e_i J^2(A)e_{i-n+1}} 
    	={\rm N}_A(i ,\, i-n+1)
        	= 0,
$
we have
$$
	{\rm N}_{T_\alpha(A)}(i ,\, i-n+1)
         = {\rm dim}_K 
          		\frac{e_i D(A) e_{i-n+1}}
             		  {e_i(J(A)D(A) + D(A) J(A) + \alpha(J(A),\,J(A))) e_{i-n+1}}.
$$
Since
$
\alpha(x_{i-2n+1},\, x_{i-2n+2} \cdots x_{i-n})
	= k(x_{i-n+1} x_{i-n+2} \cdots x_{i-1})^*
$
is in $\alpha(J(A),\,J(A))$,
it is easy to see that
\begin{align*}
	e_i D(A) e_{i-n+1}
    	&= K(x_{i-n+1} x_{i-n+2} \cdots x_{i-1})^* \\
        	&= e_i(J(A)D(A) + D(A) J(A) + \alpha(J(A),\,J(A))) e_{i-n+1}.
\end{align*}
Then we have ${\rm N}_{T_\alpha(A)}(i ,\, i-n+1) = 0$.
\end{proof}
%
%
\begin{corollary}\label{corollary1}
	Suppose that $n \geq 2$ and $A=K\Delta/R_\Delta^n$.
    Let $\alpha : A \times A \rightarrow D(A)$
    be a $2$-cocycle
    and $[\alpha] = \sum_{q=n}^{2n-1} [\beta_q]$,
    where $\beta_q : A \times A \rightarrow D(A)$
    is a $2$-cocycle such that the cohomology class
    $[\beta_q]$ of $\beta_q$
    belongs to $\Theta(D(HH_{2,\,q}(A)))$.
    Then the following equation holds{\rm :}
    $$
    \Delta_{T_\alpha(A)}=
    \begin{cases}
 		\Delta_{T_0(A)}  & \text{if $[\beta_{2n-1}] = 0$}, \\
        \Delta & \text{if $[\beta_{2n-1}] \neq 0$}.
	\end{cases}
    $$
\end{corollary}
\begin{proof}
In the proof of Theorem \ref{Main theorem},
we consider the three cases.
Then, in {\it Case $1$} and {\it Case $2$},
$\Delta_{T_\alpha(A)}$ coincides with $\Delta_{T_0(A)}$,
and in {\it Case $3$}, $\Delta_{T_\alpha(A)}$ coincides with $\Delta$.

By the proof of Theorem \ref{Main theorem},
$\beta_q$ corresponds to a 2-cocycle of {\it Case} $1$,{\it Case} $2$ or {\it Case} $3$
for each $q \, (n \leq q \leq 2n-1$).
If $q = n$, then there exist $k_i \in K\,(1 \leq i \leq s)$ such that
\begin{align*}
	[\beta_n] 
    	&=\begin{cases}
        	\sum_{i=1}^{s-1} k_i [\alpha_i] = [\sum_{i=1}^{s-1} k_i \alpha_i] & \text{if ${\rm char}(K) \nmid (n/s)$,} \\
            \sum_{i=1}^{s} k_i [\alpha_i] = [\sum_{i=1}^{s} k_i \alpha_i] & \text{if ${\rm char}(K) | (n/s)$,}
        \end{cases} \\
\end{align*}
where $\alpha_i$ is the 2-cocycle in {\it Case $2$}.
We put $\gamma_n := \sum_i k_i \alpha_i$.
If $n+1 \leq q \leq 2n-1$, then there exists $k_q \in K$ such that
$
	[\beta_q] 
    	= k_q [\sum_{i=1}^{s} \alpha_i]
        = [k_q \sum_{i=1}^{s} \alpha_i],
$
where $\sum_{i=1}^{s} \alpha_i$ is the 2-cocycle in {\it Case $1$} or {\it Case $3$}.
We put $\gamma_q := k_q \sum_{i=1}^{s} \alpha_i$.
Then it is easy to see that
$$
	\Delta_{T_{\alpha}(A)} 
    	= \Delta_{T_{\sum_{q=n}^{2n-1}\beta_q}(A)}
        	= \Delta_{T_{\sum_{q=n}^{2n-1}\gamma_q}(A)}.
$$

If $[\beta_{2n-1}] = 0$, we have
$
	\Delta_{T_{\alpha}(A)} 
		= \Delta_{T_{\sum_{q=n}^{2n-2}\beta_q}(A)}
        	= \Delta_{T_{\sum_{q=n}^{2n-2}\gamma_q}(A)}$.
Since $\sum_{q=n}^{2n-2}\gamma_q$ satisfies the conditions of Lemma \ref{key lemma},
it follows that $\Delta_{T_{\alpha}(A)} = \Delta_{T_0(A)}$.

If $[\beta_{2n-1}] \neq 0$,
then in a similar way to {\it Case $3$} in the proof of Theorem \ref{Main theorem},
there is no arrow from $i$ to $i-n+1$ in $\Delta_{T_{\alpha}(A)}$
for each $i$ $(1 \leq i \leq s)$.
Hence, we have $\Delta_{T_{\alpha}(A)} = \Delta$.
\end{proof}
%
%
\begin{corollary}\label{corollary2}
	Suppose that $n \geq 2$ and $A=K\Delta/R_\Delta^n$.
	Let $\alpha : A \times A \rightarrow D(A)$
    be a $2$-cocycle.
    If $\Delta_{T_\alpha(A)} = \Delta$, then
    $T_\alpha(A)$ is isomorphic to $K\Delta/R_\Delta^{2n}$
    and $T_\alpha(A)$ is symmetric.
\end{corollary}
\begin{proof}
Since $T_\alpha(A)$ is the self-injective Nakayama algebra,
by \cite[Section 3.1]{Erdmann Holm},
$T_\alpha(A)$ is isomorphic to $K\Delta/R_\Delta^m$
for some $m \geq 2$.
Since 
${\rm dim}_K\, T_\alpha(A) = 2\,{\rm dim}_K\, A$,
we have $m=2n$.
In the same way to Corollary \ref{corollary1},
let $[\alpha] = \sum_{q=n}^{2n-1} [\beta_q]$.
Since $\Delta_{T_\alpha(A)} = \Delta$,
we have $[\beta_{2n-1}] \neq 0$.
By (\ref{2nd Hochschild of truncated}) and the proof of Theorem \ref{Main theorem},
$2n-1$ is divided by $s$, that is, $2n \equiv 1\, ({\rm mod}\,s)$.
By \cite[Section 4.1]{Erdmann Holm},
$T_\alpha(A)$ is symmetric.
\end{proof}
%
%
\section{Examples}
Let $K$ be a field.
We consider the case $s=3$, that is,
$\Delta$ is the following quiver:
$$
{\unitlength 0.1in
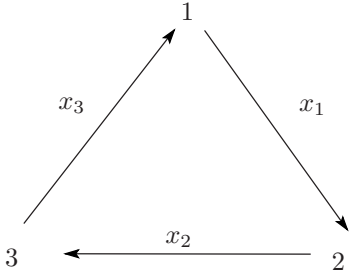
\begin{picture}( 17.9100, 14.1000)( 44.2000,-25.5000)
%
{\color[named]{Black}{%
\special{pn 8}%
\special{pa 5466 1356}%
\special{pa 6212 2398}%
\special{fp}%
\special{sh 1}%
\special{pa 6212 2398}%
\special{pa 6188 2332}%
\special{pa 6180 2356}%
\special{pa 6156 2356}%
\special{pa 6212 2398}%
\special{fp}%
}}%
%
{\color[named]{Black}{%
\special{pn 8}%
\special{pa 6018 2526}%
\special{pa 4736 2524}%
\special{fp}%
\special{sh 1}%
\special{pa 4736 2524}%
\special{pa 4804 2544}%
\special{pa 4790 2524}%
\special{pa 4804 2504}%
\special{pa 4736 2524}%
\special{fp}%
}}%
%
{\color[named]{Black}{%
\special{pn 8}%
\special{pa 4520 2364}%
\special{pa 5306 1352}%
\special{fp}%
\special{sh 1}%
\special{pa 5306 1352}%
\special{pa 5250 1392}%
\special{pa 5274 1394}%
\special{pa 5282 1418}%
\special{pa 5306 1352}%
\special{fp}%
}}%
\put(53.4000,-13.0000){\makebox(0,0)[lb]{$1$}}%
\put(61.3000,-26.1000){\makebox(0,0)[lb]{$2$}}%
\put(44.2000,-25.9000){\makebox(0,0)[lb]{$3$}}%
\put(59.6000,-18.0000){\makebox(0,0)[lb]{$x_1$}}%
\put(52.6000,-24.9000){\makebox(0,0)[lb]{$x_2$}}%
\put(47.0000,-17.9000){\makebox(0,0)[lb]{$x_3$}}%
\end{picture}}%

$$
and $A= K\Delta/R_{\Delta}^n$.
By (\ref{2nd Hochschild of truncated}),
$HH_{2,\,q}(A) \neq 0$
if and only if
$n \leq q \leq 2n-1$ and $3|q$.

%
%
First, we consider the case $n=4$.
Then $HH_{2,q} (A) \neq0$
if and only if $q=6$.
By (\ref{case1}), we have
$$
	D(HH_{2,6}(A)) 
    	= \langle
        	\sum_{i=1}^{3} (x_{i+4}x_{i+5} \otimes_{K\Delta_0^e} x_{i}x_{i+1}x_{i+2}x_{i+3})^*
        \rangle.
$$
For $a,\,b \in \Delta_+$,
the $2$-cocycle
$\alpha : A \times A \rightarrow D(A)$
corresponding to
$\sum_{i=1}^{3} (x_{i+4}x_{i+5} \otimes_{K\Delta_0^e} x_{i}x_{i+1}x_{i+2}x_{i+3})^*$
is given by
\begin{align*}
	\alpha(a,\,b) 
    	= \begin{cases}
    		(x_{i+4}x_{i+5})^*  & \text{if $ab = x_{i}x_{i+1}x_{i+2}x_{i+3}$}, \\
            x_{i+5}^*           & \text{if $ab = x_{i}x_{i+1}x_{i+2}x_{i+3}x_{i+4}$}, \\
            e_{i+6}^*             & \text{if $ab = x_{i}x_{i+1}x_{i+2}x_{i+3}x_{i+4}x_{i+5}$}, \\
            0 & \text{otherwise}.
    	\end{cases}
\end{align*}
By \cite[Proposition 2.2]{Fernandez}, $\Delta_{T_0(A)}$ is the following quiver:
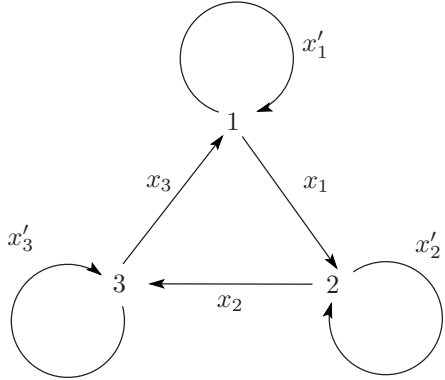
\begin{center}
{\unitlength 0.1in
\begin{picture}( 22.7600, 19.6600)( 35.7000,-22.3000)
%
{\color[named]{Black}{%
\special{pn 8}%
\special{pa 4800 970}%
\special{pa 5292 1656}%
\special{fp}%
\special{sh 1}%
\special{pa 5292 1656}%
\special{pa 5270 1590}%
\special{pa 5262 1614}%
\special{pa 5238 1614}%
\special{pa 5292 1656}%
\special{fp}%
}}%
%
{\color[named]{Black}{%
\special{pn 8}%
\special{pa 5164 1742}%
\special{pa 4318 1740}%
\special{fp}%
\special{sh 1}%
\special{pa 4318 1740}%
\special{pa 4386 1760}%
\special{pa 4372 1740}%
\special{pa 4386 1720}%
\special{pa 4318 1740}%
\special{fp}%
}}%
%
{\color[named]{Black}{%
\special{pn 8}%
\special{pa 4176 1634}%
\special{pa 4694 968}%
\special{fp}%
\special{sh 1}%
\special{pa 4694 968}%
\special{pa 4638 1008}%
\special{pa 4662 1010}%
\special{pa 4670 1032}%
\special{pa 4694 968}%
\special{fp}%
}}%
\put(47.1700,-9.3200){\makebox(0,0)[lb]{$1$}}%
\put(52.3800,-17.9000){\makebox(0,0)[lb]{$2$}}%
\put(41.2200,-17.8300){\makebox(0,0)[lb]{$3$}}%
\put(51.1600,-12.6200){\makebox(0,0)[lb]{$x_1$}}%
\put(46.6400,-18.8200){\makebox(0,0)[lb]{$x_2$}}%
\put(42.9400,-12.5600){\makebox(0,0)[lb]{$x_3$}}%
%
{\color[named]{Black}{%
\special{pn 8}%
\special{ar 4770 560 296 296  1.9013767  1.1863478}%
}}%
%
{\color[named]{Black}{%
\special{pn 8}%
\special{pa 4890 832}%
\special{pa 4882 836}%
\special{fp}%
\special{sh 1}%
\special{pa 4882 836}%
\special{pa 4952 834}%
\special{pa 4932 818}%
\special{pa 4938 796}%
\special{pa 4882 836}%
\special{fp}%
}}%
%
{\color[named]{Black}{%
\special{pn 8}%
\special{ar 3886 1934 296 296  6.0288849  5.3170404}%
}}%
%
{\color[named]{Black}{%
\special{pn 8}%
\special{pa 4048 1684}%
\special{pa 4056 1690}%
\special{fp}%
\special{sh 1}%
\special{pa 4056 1690}%
\special{pa 4014 1634}%
\special{pa 4012 1658}%
\special{pa 3990 1666}%
\special{pa 4056 1690}%
\special{fp}%
}}%
%
{\color[named]{Black}{%
\special{pn 8}%
\special{ar 5550 1920 296 296  4.1077376  3.3958931}%
}}%
%
{\color[named]{Black}{%
\special{pn 8}%
\special{pa 5260 1860}%
\special{pa 5264 1846}%
\special{fp}%
\special{sh 1}%
\special{pa 5264 1846}%
\special{pa 5226 1906}%
\special{pa 5250 1898}%
\special{pa 5266 1916}%
\special{pa 5264 1846}%
\special{fp}%
}}%
\put(57.0300,-16.0600){\makebox(0,0)[lb]{$x'_2$}}%
\put(51.1200,-5.5500){\makebox(0,0)[lb]{$x'_1$}}%
\put(35.7000,-15.7100){\makebox(0,0)[lb]{$x'_3$}}%
\end{picture}}%

\end{center}
Then we know that $T_{\alpha}(A)$ is isomorphic to
$K\Delta_{T_0(A)}/I$,
where
\begin{align*}
	I 
    	= \langle
            x'_i x_i - x_ix'_{i+1},\,
            x_i x_{i+1} x_{i+2} x_{i+3} - x'_i x_i,\,
            (x'_i)^2
            \,|\, i=1,\,2,\,3
		\rangle.
\end{align*}
On the other hand,
$T_0(A)$ is isomorphic to
$K\Delta_{T_0(A)}/I_0$,
where
\begin{align*}
	I_0 
    	= \langle
			x'_i x_i - x_ix'_{i+1},\,
            x_i x_{i+1} x_{i+2} x_{i+3},\,
            (x'_i)^2 \,
            | \, i=1,\,2,\,3
		\rangle.
\end{align*}

%
%
Second, we consider the case $n=3$.
Then $HH_{2,q} (A) \neq0$
if and only if $q=3$.
By (\ref{case2}), we have
$$
	D(HH_{2,3}(A)) = \langle (e_1 \otimes_{K\Delta_0^e} x_1x_2x_3)^*,\, (e_2 \otimes_{K\Delta_0^e} x_2x_3x_1)^* \rangle.
$$
For $a,\,b \in \Delta_+$,
the $2$-cocycle
$\beta_1 : A \times A \rightarrow D(A)$
corresponding to $(e_1 \otimes_{K\Delta_0^e} x_1x_2x_3)^*$ is given by
$$
	\beta_1(a,\,b) 
    	= \begin{cases}
    		e_1^*  & \text{if $ab = x_1x_2x_3$}, \\
            0 & \text{otherwise},
    	\end{cases}
$$
and $\beta_2 : A \times A \rightarrow D(A)$
corresponding to $(e_2 \otimes_{K\Delta_0^e} x_2x_3x_1)^*$ is given by
$$
	\beta_2(a,\,b)
    	= \begin{cases}
    		e_2^*  & \text{if $ab = x_2x_3x_1$}, \\
            0 & \text{otherwise}.
    	\end{cases}
$$
For any $2$-cocycle $\beta := k_1\beta_1 + k_2\beta_2 \,(k_1,\,k_2 \in K)$,
by Theorem \ref{Main theorem}, we have $\Delta_{T_{\beta}(A)} = \Delta_{T_0(A)}$.
By \cite[Proposition 2.2]{Fernandez}, $\Delta_{T_0(A)}$ is the following quiver:
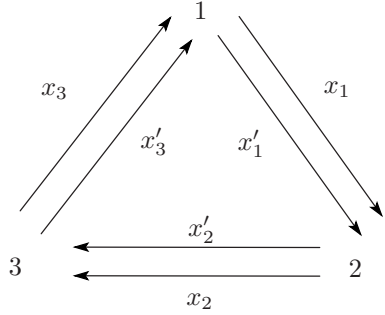
\begin{align*}
{\unitlength 0.1in
\begin{picture}( 19.5100, 15.3000)( 43.9000,-26.7000)
%
{\color[named]{Black}{%
\special{pn 8}%
\special{pa 5596 1286}%
\special{pa 6342 2328}%
\special{fp}%
\special{sh 1}%
\special{pa 6342 2328}%
\special{pa 6318 2262}%
\special{pa 6310 2286}%
\special{pa 6286 2286}%
\special{pa 6342 2328}%
\special{fp}%
}}%
%
{\color[named]{Black}{%
\special{pn 8}%
\special{pa 6018 2646}%
\special{pa 4736 2644}%
\special{fp}%
\special{sh 1}%
\special{pa 4736 2644}%
\special{pa 4804 2664}%
\special{pa 4790 2644}%
\special{pa 4804 2624}%
\special{pa 4736 2644}%
\special{fp}%
}}%
%
{\color[named]{Black}{%
\special{pn 8}%
\special{pa 4450 2304}%
\special{pa 5236 1292}%
\special{fp}%
\special{sh 1}%
\special{pa 5236 1292}%
\special{pa 5180 1332}%
\special{pa 5204 1334}%
\special{pa 5212 1358}%
\special{pa 5236 1292}%
\special{fp}%
}}%
\put(53.6000,-13.0000){\makebox(0,0)[lb]{$1$}}%
\put(61.7000,-26.5000){\makebox(0,0)[lb]{$2$}}%
\put(43.9000,-26.4000){\makebox(0,0)[lb]{$3$}}%
\put(60.4000,-17.1000){\makebox(0,0)[lb]{$x_1$}}%
\put(53.2000,-28.2000){\makebox(0,0)[lb]{$x_2$}}%
\put(45.6000,-17.2000){\makebox(0,0)[lb]{$x_3$}}%
%
{\color[named]{Black}{%
\special{pn 8}%
\special{pa 6018 2496}%
\special{pa 4736 2494}%
\special{fp}%
\special{sh 1}%
\special{pa 4736 2494}%
\special{pa 4804 2514}%
\special{pa 4790 2494}%
\special{pa 4804 2474}%
\special{pa 4736 2494}%
\special{fp}%
}}%
%
{\color[named]{Black}{%
\special{pn 8}%
\special{pa 4560 2424}%
\special{pa 5346 1412}%
\special{fp}%
\special{sh 1}%
\special{pa 5346 1412}%
\special{pa 5290 1452}%
\special{pa 5314 1454}%
\special{pa 5322 1478}%
\special{pa 5346 1412}%
\special{fp}%
}}%
%
{\color[named]{Black}{%
\special{pn 8}%
\special{pa 5486 1386}%
\special{pa 6232 2428}%
\special{fp}%
\special{sh 1}%
\special{pa 6232 2428}%
\special{pa 6208 2362}%
\special{pa 6200 2386}%
\special{pa 6176 2386}%
\special{pa 6232 2428}%
\special{fp}%
}}%
\put(53.3000,-24.7000){\makebox(0,0)[lb]{$x_2'$}}%
\put(50.8000,-20.1000){\makebox(0,0)[lb]{$x_3'$}}%
\put(55.9000,-20.2000){\makebox(0,0)[lb]{$x'_1$}}%
\end{picture}}%

\end{align*}
~\\
Then we know that $T_{\beta}(A)$ is isomorphic to
$K\Delta_{T_0(A)}/I'$,
where
\begin{align*}
	I' 
    	= \langle
				x_i x'_{i+1} -& x'_i x_{i+1}, \,
                x'_i x'_{i+1}, \\
                &x_1 x_2 x_3 - k_1 x_1 x_2 x'_3 ,\,
                x_2 x_3 x_1 - k_2 x_2 x_3 x'_1 ,\,
                x_3 x_1 x_2 
                \, | \, i=1,\,2,\,3
		\rangle.
\end{align*}
On the other hand,
$T_0(A)$ is isomorphic to
$K\Delta_{T_0(A)}/I'_0$,
where 
\begin{align*}
	I'_0 
    	= \langle
				x_i x'_{i+1} - x'_i x_{i+1}, \,
                x'_i x'_{i+1},\,
                x_i x_{i+1} x_{i+2}
                \, | \, i=1,\,2,\,3
		\rangle.
\end{align*}
Moreover,
it is easy to see that $T_{\beta_1}(A)$ and $T_{\beta_2}(A)$
are isomorphic as $K$-algebras.
However, they are not equivalent as Hochschild extensions.

%
%
And finally, we consider the case $n=2$.
Then $HH_{2,q}(A) \neq 0$ if and only if $q=n+1=3$.
By (\ref{case3}), we have
$$
		D(HH_{2,3}(A)) = \langle (x_3 \otimes_{K\Delta_0^e} x_1x_2)^*+(x_1 \otimes_{K\Delta_0^e} x_2x_3)^*+(x_2 \otimes_{K\Delta_0^e} x_3x_1)^* \rangle.
$$
The $2$-cocycle $\gamma : A \times A \rightarrow D(A)$
corresponding to $(x_3 \otimes_{K\Delta_0^e} x_1x_2)^*+(x_1 \otimes_{K\Delta_0^e} x_2x_3)^*+(x_2 \otimes_{K\Delta_0^e} x_3x_1)^*$ is given by
$$
	\gamma (a,\,b) 
    	= \begin{cases}
        	x_i^* & \text{if $ab = x_{i+1}x_{i+2}$}, \\
            0     & \text{otherwise},
        \end{cases}
$$
for $a,\,b \in \Delta_+$.
Then, we have
$\Delta_{T_{\gamma}(A)} = \Delta$.
It is easy to see that $T_\gamma(A)$
is isomorphic to $K\Delta/R_{\Delta}^4$.
By Corollary \ref{corollary2},
$T_\gamma(A)$ is symmetric.
We remark that $T_\gamma(A)$ satisfies the condition of \cite[Theorem 1]{yamagata}.
On the other hand,
by \cite[Proposition 2.2]{Fernandez},
$\Delta_{T_0(A)}$ is the following quiver:
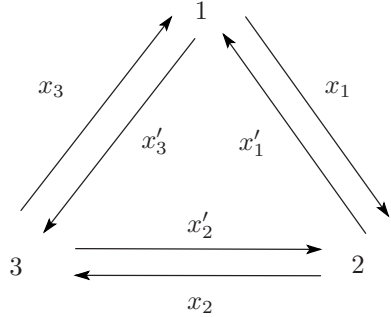
\begin{align*}
{\unitlength 0.1in
\begin{picture}( 19.8200, 15.4000)( 43.9000,-26.8000)
%
{\color[named]{Black}{%
\special{pn 8}%
\special{pa 5626 1290}%
\special{pa 6372 2332}%
\special{fp}%
\special{sh 1}%
\special{pa 6372 2332}%
\special{pa 6350 2266}%
\special{pa 6342 2290}%
\special{pa 6318 2290}%
\special{pa 6372 2332}%
\special{fp}%
}}%
%
{\color[named]{Black}{%
\special{pn 8}%
\special{pa 6018 2646}%
\special{pa 4736 2644}%
\special{fp}%
\special{sh 1}%
\special{pa 4736 2644}%
\special{pa 4804 2664}%
\special{pa 4790 2644}%
\special{pa 4804 2624}%
\special{pa 4736 2644}%
\special{fp}%
}}%
%
{\color[named]{Black}{%
\special{pn 8}%
\special{pa 4450 2304}%
\special{pa 5236 1292}%
\special{fp}%
\special{sh 1}%
\special{pa 5236 1292}%
\special{pa 5180 1332}%
\special{pa 5204 1334}%
\special{pa 5212 1358}%
\special{pa 5236 1292}%
\special{fp}%
}}%
\put(53.6000,-13.0000){\makebox(0,0)[lb]{$1$}}%
\put(61.8000,-26.3000){\makebox(0,0)[lb]{$2$}}%
\put(43.9000,-26.4000){\makebox(0,0)[lb]{$3$}}%
\put(60.4000,-17.1000){\makebox(0,0)[lb]{$x_1$}}%
\put(53.2000,-28.4000){\makebox(0,0)[lb]{$x_2$}}%
\put(45.4000,-17.1000){\makebox(0,0)[lb]{$x_3$}}%
%
{\color[named]{Black}{%
\special{pn 8}%
\special{pa 4736 2504}%
\special{pa 6016 2508}%
\special{fp}%
\special{sh 1}%
\special{pa 6016 2508}%
\special{pa 5950 2488}%
\special{pa 5964 2508}%
\special{pa 5950 2528}%
\special{pa 6016 2508}%
\special{fp}%
}}%
%
{\color[named]{Black}{%
\special{pn 8}%
\special{pa 5360 1404}%
\special{pa 4572 2416}%
\special{fp}%
\special{sh 1}%
\special{pa 4572 2416}%
\special{pa 4630 2376}%
\special{pa 4606 2374}%
\special{pa 4598 2352}%
\special{pa 4572 2416}%
\special{fp}%
}}%
%
{\color[named]{Black}{%
\special{pn 8}%
\special{pa 6254 2424}%
\special{pa 5510 1382}%
\special{fp}%
\special{sh 1}%
\special{pa 5510 1382}%
\special{pa 5532 1448}%
\special{pa 5542 1424}%
\special{pa 5566 1424}%
\special{pa 5510 1382}%
\special{fp}%
}}%
\put(53.2000,-24.5000){\makebox(0,0)[lb]{$x_2'$}}%
\put(50.8000,-20.1000){\makebox(0,0)[lb]{$x_3'$}}%
\put(55.9000,-20.2000){\makebox(0,0)[lb]{$x_1'$}}%
\end{picture}}%

\end{align*}
~\\
and
$T_0(A)$ is isomorphic to $K\Delta_{T_0(A)}/I_0''$,
where 
$$
 I''_0 = \langle
 			x_i x'_i - x'_{i+2} x_{i+2},\,
            x_i x_{i+1},\,
            x'_i x'_{i-1}
            \, | \, i=1,\,2,\,3
        \rangle.
$$

\section*{Acknowledgment}
The authors thank Professor Takahiko Furuya for many valuable comments
and suggestions.

%


\end{document}